\documentclass[a4paper]{amsart}
\usepackage{tikz}
\usepackage{tikz-cd}
\usepackage[all]{xy}
\usepackage{amsfonts}
\usepackage{amsmath}
\usepackage{amsthm}
\usepackage{amssymb}
\usepackage{amscd}
\usepackage{comment}
\usepackage{graphicx}
\usepackage[UKenglish]{babel}

\usepackage{mathrsfs}       

\usepackage{colonequals} 
\usepackage{comment}

\usepackage{color}\definecolor{darkblue}{rgb}{0,0.1,.5}
\usepackage[colorlinks=true,linkcolor=darkblue, urlcolor=darkblue, citecolor=darkblue]{hyperref}
\usepackage{cleveref}

\def\le{\leqslant}
\def\ge{\geqslant}

\def\C{\mathbb{C}}
\def\Q{\mathbb{Q}}
\def\Z{\mathbb{Z}}

\newcommand{\mb}[1]{{\textbf {\textit#1}}}

\DeclareMathOperator{\Hom}{Hom}
\DeclareMathOperator{\Tor}{Tor}

\DeclareMathOperator\MF{MF}

\def\sk{\mathcal K}
\def\sK{\mathcal K}

\def\ts{\textsc}

\def\pt{\mathit{pt}}

\def\zk{\mathcal Z_{\mathcal K}}

\newcommand{\hr}[2][]{\hyperref[#2]{#1~\ref{#2}}}

\makeatletter
\newtheorem*{rep@theorem}{\rep@title}
\newcommand{\newreptheorem}[2]{%
\newenvironment{rep#1}[1]{%
 \def\rep@title{#2 \ref{##1}}%
 \begin{rep@theorem}}%
 {\end{rep@theorem}}}
\makeatother

\newreptheorem{theorem}{Theorem}

\newtheorem{theorem}{Theorem}[section]
\newtheorem{proposition}[theorem]{Proposition}
\newtheorem{lemma}[theorem]{Lemma}

\theoremstyle{definition}
\newtheorem{example}[theorem]{Example}

\numberwithin{equation}{section}

\makeatletter
\@namedef{subjclassname@2020}{\textup{2020} Mathematics Subject Classification}
\makeatother

\title[Equivariant cohomology of moment-angle complexes]{Equivariant cohomology of moment-angle complexes with respect to coordinate subtori}

\author{Taras Panov}
\address{Faculty of Mathematics and Mechanics, Lomonosov Moscow
State University, Russia;\newline
Faculty of Computer Science, National Research University Higher School of Economics, Moscow, Russia;\newline
Institute for Information Transmission Problems, Russian Academy of Sciences, Moscow, Russia;\newline
Steklov Mathematical Institute of Russian Academy of Sciences, Moscow, Russia}
\email{tpanov@mech.math.msu.su}

\author{Indira Zeinikesheva}
\address{Faculty of Mathematics and Mechanics, Lomonosov Moscow
State University, Russia;\newline
Steklov Mathematical Institute of Russian Academy of Sciences, Moscow, Russia}
\email{znikzk@gmail.com}



\thanks{Research of T. Panov was carried out at the Steklov Mathematical Institute and funded by the Russian Science Foundation, grant 20-11-19998.}
\thanks{Research of I. Zeinikesheva was carried out at the Steklov International Mathematical Center and supported by the Ministry of Science and Higher Education of the Russian Federation (agreement no. 075-15-2022-265).}

\begin{document}

\maketitle

\begin{abstract}
We compute the equivariant cohomology $H^*_{T_I}(\zk)$ of moment-angle complexes $\zk$ with respect to the action of coordinate subtori $T_I \subset T^m$. We give a criterion for the equivariant formality of $\zk$ and obtain specifications for the cases of flag complexes and graphs. 
\end{abstract}

\section{Introduction}
Let $\sK$ be a simplicial complex on an $m$-element set $V$, and let $\zk$ be the corresponding moment-angle complex~$\zk$. 
We study the equivariant cohomology of $\zk$ with respect to the action of coordinate subtori $T_I\subset T^m$, where $I=\{i_1,\ldots,i_k\}\subset V$.

We construct two commutative integral dga models for $H^*_{T_I}(\zk)$. The first is given by
\[
  \bigl(\Lambda[u_i\colon i\notin I]\otimes\Z[\sK],d\bigr),\quad du_i=v_i,\; dv_i=0, 
\]
where $\Lambda[u_i\colon i\notin I]$ is the exterior algebra on degree-one generators $u_i$, $i\notin I$, and $\Z[\sK]$ is the face ring of~$\zk$. The second dga model $R_I(\sK)$ is given by the quotient of the first one by the ideal generated by $u_iv_i$ and $v_i^2$  with $i\notin I$. 
As a result we obtain
\begin{reptheorem}{cohomology}
There are isomorphisms of rings 
\begin{align*}
  H^*_{T_I}(\zk)&\cong H\bigl(\Lambda[u_i\colon i\notin I]\otimes\Z[\sK],d\bigr)
  \cong H^*\bigl(R_I(\sK)\bigr)\\
  &\cong\Tor_{\Z[v_1,\ldots,v_m]}\bigl(\Z[v_i\colon i\in I],\Z[\sK]\bigr),
\end{align*}
where $\Z[v_i\colon i\in I]$ is the $\Z[v_1,\ldots,v_m]$-module via the homomorphism sending $v_i$ to $0$ for $i\notin I$.
\end{reptheorem}

When $I=V$, the dga model above reduces to the face ring $\Z[\sK]$ with zero differential, and we recover the integral formality result of~\cite{no-ra05}.

When $I=\varnothing$, Theorem~\ref{cohomology} gives the description of the ordinary integral cohomology of 
$\zk$ of~\cite{b-b-p04} and~\cite{fran06}.

The additive (or $\Z[v_1,\ldots,v_m]$-module) isomorphism 
\[
H^*_{T_I}(\zk)\cong\Tor_{\Z[v_1,\ldots,v_m]}\bigl(\Z[v_i\colon i\in I],\Z[\sK]\bigr)\cong
\Tor_{\Z[v_i\colon i\notin I]}\bigl(\Z,\Z[\sK]\bigr)
\]
follows from the result of~\cite{l-m-m14}.

\smallskip

Next, we study the equivariant formality of $\zk$, that is, whether $H^*_{T_I}(\zk)$ is a free module over the polynomial ring $H^*_{T_I}(pt) = H^*(BT_I)=\Z[v_i\colon i\in I]$. 
We prove
\begin{reptheorem}{suffic}
Let $\sK$ be a simplicial complex on a finite set $V$. The following conditions are equivalent:

\begin{itemize}
\item[(a)] For any $I\in\sK$, the equivariant cohomology $H^*_{T_I}(\mathcal{Z}_\mathcal{K})$ is a free module over $H^*(BT_I)$.

\smallskip

\item[(b)] There is a partition $V=V_1\sqcup\cdots\sqcup V_p\sqcup U$ such that
\[
  \sK=\partial\varDelta(V_1)\mathbin{*}\cdots
  \mathbin{*}\partial\varDelta(V_p)\mathbin{*}
  \varDelta(U),
\]
where $\varDelta(U)$ denotes a full simplex on $U$, and $\partial\varDelta(V_i)$ denotes the boundary of a simplex on~$V_i$.

\smallskip

\item[(c)] The rational face ring $\mathbb Q[\sK]$ is a complete intersection ring (the quotient of the polynomial ring by an ideal generated by a regular sequence).
\end{itemize}
\end{reptheorem}

In the case of flag complexes we have the following specification:
\begin{reptheorem}{flag}
Let $\mathcal{K}$ be a flag complex on~$V$. Then the following conditions are equivalent:
\begin{itemize}
\item[(a)]  $H^*_{S^1_i}(\mathcal{Z}_\mathcal{K})$ is a free module over $\Z[v_i]$ for all $i$.

\item[(b)]$\sK=\partial\varDelta(V_1)\mathbin{*}\cdots\mathbin{*}
\partial\varDelta(V_p)\mathbin{*} \varDelta(U)$ where $|V_k|=2$ for $k=1,\ldots,p$.
\end{itemize}
\end{reptheorem}

A similar criterion holds in the case when $\sK$ is a simple graph:
\begin{reptheorem}{graph}
Let $\mathcal{K}$ be a one-dimensional complex (a simple graph). Then the following conditions are equivalent: 
\begin{itemize}
\item[(a)]
$H^*_{S^1_i}(\mathcal{Z}_{\mathcal{K}})$ is a free module over $\mathbb{Z}[\upsilon_i]$ for any $i$.
\item[(b)]
$\mathcal{K}$ is the one of the following: $\partial\varDelta^2$, $\partial\varDelta^1\mathbin{*}\partial\varDelta^1$, $\partial\varDelta^1$, $\varDelta^1$, $\partial\varDelta^1*\varDelta^0$, $\varDelta^0$.
\end{itemize}
\end{reptheorem}

Along the way we establish some additional properties of the equivariant cohomology of $\zk$ and give illustrative examples.

\section{Preliminaries}\label{sec_prel}
Let $\sK$ be a simplicial complex on a finite $m$-element set $V$, which we often identify with the index set $[m]=\{1,2,\ldots,m\}$. We refer to a subset $I=\{i_1,\ldots,i_k\}\subset V$ that is contained in $\sK$ as a \emph{simplex}.  We assume that $\varnothing\in\sK$ and allow \emph{ghost vertices}, that is, one-element subsets 
$\{i\}\in V$ such that $\{i\}\notin\sK$. 

Let $(\mb X,\mb A)=\{(X_1,A_1),\ldots,(X_m,A_m)\}$ be a sequence of $m$ pairs of pointed CW-complexes, $A_i\subset X_i$. For each
subset $I\subset V$, define
\[
  (\mb X,\mb A)^I\colonequals\bigl\{(x_1,\ldots,x_m)\in
  \prod_{j=1}^m X_j\colon\; x_j\in A_j\quad\text{for }j\notin I\bigl\}.
\]
The \emph{polyhedral product} of $(\mb X,\mb A)$ corresponding to $\sK$ is 
\[
  (\mb X,\mb A)^{\sK}\colonequals\bigcup_{I\in\mathcal K}(\mb X,\mb A)^I=
  \bigcup_{I\in\mathcal K}
  \Bigl(\prod_{i\in I}X_i\times\prod_{i\notin I}A_i\Bigl).
\]
Using the categorical language, denote by $\ts{cat}(\sk)$ the face category of $\sk$, with objects $I\in\sK$ and morphisms $I\subset J$. Define the
$\ts{cat}(\sK)$-diagram
\[
\begin{aligned}
  \mathcal D_\sK(\mb X,\mb A)\colon \ts{cat}(\sK)&\longrightarrow \ts{top},\\
  I&\longmapsto (\mb X,\mb A)^I,
\end{aligned}
\]
which maps the morphism $I\subset J$ of $\ts{cat}(\sK)$ to the
inclusion of spaces $(\mb X,\mb A)^I\subset(\mb X,\mb A)^J$. Then
we have
\[
  (\mb X,\mb A)^\sK=\mathop{\mathrm{colim}}
  \mathcal D_\sK(\mb X,\mb A)=\mathop{\mathrm{colim}}_{I\in\sK}
  (\mb X,\mb A)^I.
\]
In the case when all the pairs $(X_i,A_i)$ are the same, i.\,e.
$X_i=X$ and $A_i=A$ for $i=1,\ldots,m$, we use the notation
$(X,A)^\sK$ for $(\mb X,\mb A)^\sK$.

The \emph{moment-angle complex} $\zk$ is the polyhedral product $(D^2,S^1)^\sK$.
We refer to \cite[Chapter~4]{bu-pa15} for more details and examples. 

The \emph{face ring} of $\sK$ is the quotient ring
\[
  \mathbb{Z}[\sk]\colonequals \mathbb{Z}[v_1, \dots, v_m]/\mathcal{I}_\sk, 
\]
where $\mathcal{I}_\sk$ is the ideal generated by the square-free monomials $v_I=\prod_{i\in I} v_i$ for which $I\subset V$ is not a simplex of $\sK$.

\section{Equivariant cohomology}

For an action of a topological group $G$ on a space $X$, the \emph{Borel construction} is
\[
  EG\times_G X\colonequals EG\times X/(e\cdot g^{-1},g\cdot x)\sim(e,x),
\]
where $EG$ is the universal right $G$-space, $e\in EG$, $g\in G$, $x\in X$. There is the \emph{Borel fibration} $EG\times_G X\to BG$ over the classifying space $BG=EG/G$ with fibre~$X$. The \emph{equivariant cohomology} of $X$ is 
\[
  H^*_G(X)\colonequals H^*(EG\times_G X).
\]

The torus $T^m=(S^1)^m$ acts on $\zk=(D^2,S^1)^\sK$ coordinatewise. The universal bundle $ES^1\to BS^1$ is the infinite-dimensional Hopf bundle $S^\infty\to\C P^\infty$. 

We consider the equivariant cohomology of $\zk$ with respect to the action of coordinate subtori
\[
  T_I = \{(t_1, \dots, t_m) \in T^m\colon t_j = 1\text{ for }j \notin I\},
\]
where $I=\{i_1,\ldots,i_k\}\subset V$.

\begin{proposition}\label{yb}
There is a homotopy equivalence
\[
  ET_I\times_{T_I}\zk\xrightarrow{\simeq}(\!\mb Y,\mb B\,)^\sK,
\]
where
\[
  Y_i=\begin{cases}\C P^\infty,& i\in I,\\ D^2,& i\notin I,  \end{cases}\qquad
  B_i=\begin{cases}\pt,& i\in I,\\ S^1,& i\notin I.  \end{cases}
\]
\end{proposition}
\begin{proof}
We have
\[
  ET_I\times_{T_I}\zk=ET_I\times_{T_I}(D^2,S^1)^\sK=(\mb X,\mb A)^\sK,
\]
where
\[
  X_i=\begin{cases}S^\infty\times_{S^1}D^2,& i\in I,\\ D^2,& i\notin I,  \end{cases}  
  \qquad
  A_i=\begin{cases}S^\infty\times_{S^1}S^1,& i\in I,\\ S^1,& i\notin I.  \end{cases}
\]
The result follows from the homotopy equivalence of pairs 
\[
  (S^\infty\times_{S^1}D^2, S^\infty\times_{S^1}S^1)
  \xrightarrow{\simeq}(\C P^\infty,\pt),
\]
as in~\cite[Theorem~4.3.2]{bu-pa15} where the case $I=[m]$ is treated.
\end{proof}

Next we introduce two commutative dga models for the equivariant cohomology $H^*_{T_I}(\zk)$. First, consider the dga
\[
  \bigl(\Lambda[u_i\colon i\notin I]\otimes\Z[\sK],d\bigr), \quad du_i=v_i,\; dv_i=0, 
\]
where $\Lambda[u_i\colon i\notin I]$ is the exterior algebra on generators indexed by $V- I$. The grading is given by $\deg u_i=1$, $\deg v_i=2$.

Second, consider the quotient dga
\[
  R_I(\sK)\colonequals\Lambda[u_i\colon i\notin I]\otimes\Z[\sK]\big/
  (u_iv_i=v_i^2=0,\;i\notin I),
\]
noting that the ideal generated by $u_iv_i$ and $v_i^2$  with $i\notin I$ is $d$-invariant.

We denote by $C_*(X)$ and $C^*(X)$ the \emph{normalised} singular chain dg coalgebra and singular cochain dg algebra of a space~$X$, respectively. (A singular cochain is normalised if it vanishes on degenerate singular simplices~\cite[\S VIII.6]{macl63}; passing to normalised cochains does not change the quasi-isomorphism type of $C^*(X)$.)

\begin{theorem}\label{quism}
The singular cochain algebra $C^*(ET_I\times_{T_I}\zk)$ is quasi-isomorphic to 
$(\Lambda[u_i\colon i\notin I]\otimes\Z[\sK],d)$ and $R_I(\sK)$.
The quasi-isomorphisms are natural with respect to inclusion of subcomplexes.
\end{theorem}
\begin{proof}
We combine the arguments of \cite{no-ra05}, \cite[\S4.5, \S8.1]{bu-pa15}  and~\cite{fran}.

The acyclicity of the ideal generated by $v_i^2$ and $u_iv_i$ for $i\notin I$ in 
$\Lambda[u_i\colon i\notin I]\otimes\Z[\sK]$ is established in the same way as \cite[Lemma~3.2.6]{bu-pa15}, where the case $I=\varnothing$ is treated. This gives a quasi-isomorphism $\Lambda[u_i\colon i\notin I]\otimes\Z[\sK]\xrightarrow{\simeq}R_I(\sK)$. For the remaining quasi-isomorphism $R_I(\sK)\simeq C^*(ET_I\times_{T_I}\zk)$, we use the homotopy equivalent polyhedral product model 
$(\!\mb Y,\mb B)^\sK$ of Proposition~\ref{yb}.

Throughout the proof, we use the following zig-zag of quasi-isomorphisms of dgas~\cite[\S7.2]{mccl01}:
\begin{equation}\label{ez}
  C^*(X)\otimes C^*(Y)\xrightarrow{\simeq}
  \Hom(C_*(X)\otimes C_*(Y),\Z)\xleftarrow{\simeq} C^*(X\times Y),
\end{equation}
where the arrow on the right is the dual of the Eilenberg--Zilber map $C_*(X)\otimes C_*(Y)\to C_*(X\times Y)$, which is a quasi-isomorphism of dg coalgebras~\cite[(17.6)]{ei-mo66}.

First consider the case $\sK=\varnothing$ with $m$ ghost vertices. Then $\zk=T^m$ and 
\[
  ET_I\times_{T_I}\zk\simeq T^m/T_I=(\!\mb Y,\mb B)^\sK=\prod_{i\notin I}S^1,
\]
whereas $R_I(\sK)=\Lambda[u_i\colon i\notin I]$. There is a quasi-isomorphism 
$\Lambda[u]=H^*(S^1)\to C^*(S^1)$ mapping $u$ to its representing singular 1-cocycle (here it is important that we work with normalised cochains). Applying~\eqref{ez} we obtain the required quasi-isomorphism $\Lambda[u_i\colon i\notin I]\simeq C^*(\prod_{i\notin I}S^1)$.


Now consider the case $m=1$ and $\sK=\varDelta^0$, a $0$-simplex. If $I=\varnothing$, then $(\!\mb Y,\mb B)^\sK=D^2$ and $R_I(\sK)=\Lambda[u]\otimes \Z[v]/(uv=v^2=0)$. Let $\varphi\colon[01]\to D^2$ be the standard parametrisation of the boundary circle $S^1$, viewed as a singular $1$-simplex. Let $\psi\colon[012]\to D^2$ be a singular $2$-simplex such that $\psi|_{[12]}=\varphi$ and $\psi|_{[02]}$, $\psi|_{[01]}$ are constant maps to the basepoint $1\in S^1$. Then $\partial\varphi=0$ and 
$\partial\psi=\varphi$, as we work with the normalised chains. Now if 
$\alpha\in C^1(D^2)$ is the cochain dual to $\varphi$ and $\beta\in C^2(D^2)$ is dual to~$\psi$, then $\Lambda[u]\otimes \Z[v]/(uv=v^2=0)\to C^*(D^2)$ mapping $u$ to
$\alpha$ and $v$ to $\beta$ is a quasi-isomorphism. If $I=\{1\}$, then $(\!\mb Y,\mb B)^\sK=\C P^\infty$ and $R_I(\sK)=\Z[v]$. There is a quasi-isomorphism 
$\Z[v]=H^*(\C P^\infty)\to C^*(\C P^\infty)$ mapping $v$ to its representing singular 2-cocycle.

Next consider the case $\sK=\varDelta^{m-1}=\varDelta[m]$, the full simplex on~$[m]$. Applying~\eqref{ez} and the K\"unneth theorem, we obtain a zig-zag of quasi-isomorphisms
\begin{multline}\label{zigzagJ}
  R_I(\varDelta[m])=\Lambda[u_i\colon i\notin I]\otimes\Z[v_1,\ldots,v_m]\big/ 
  (u_iv_i,v_i^2\colon i\notin I)
  \\=\bigotimes_{i\in I}\Z[v_i]\otimes
  \bigotimes_{i\notin I}\bigr(\Lambda[u_i]\otimes\Z[v_i]/(u_iv_i,v_i^2)\bigl)
  \xrightarrow{\simeq}\bigotimes_{i\in I}C^*(\C P^\infty)\otimes
  \bigotimes_{i\notin I}C^*(D^2)\\
  \xrightarrow{\simeq}\cdots\xleftarrow{\simeq}
  C^*\Bigl(\prod_{i\in I}\C P^\infty\times\prod_{i\notin I}D^2\Bigr)=
  C^*\bigl((\!\mb Y,\mb B)^{\varDelta[m]}\bigr),
\end{multline}
which completes the proof for the case $\sK=\varDelta[m]$.

The general case is proved by induction on the number of simplices in~$\sK$ using the naturality with respect of inclusion of subcomplexes and the Mayer--Vietoris sequence, as in~\cite[Theorem~1]{fran}. Namely, we add simplices one by one to the empty simplicial complex on $[m]$ and use the zig-zag of dga maps between the two short exact sequences for any two simplicial complexes $\sK_1$ and $\sK_2$ on~$[m]$:
\[
   0\to R_I(\sK_1\cup\sK_2) \to R_I(\sK_1)\oplus R_I(\sK_2)
   \to R_I(\sK_1\cap\sK_2) \to 0
\]
and
\[
   0\to C^*((\!\mb Y,\mb B)^{\sK_1\cup\sK_2}) 
   \to C^*((\!\mb Y,\mb B)^{\sK_1})\oplus 
   C^*((\!\mb Y,\mb B)^{\sK_2})
   \to C^*((\!\mb Y,\mb B)^{\sK_1\cap\sK_2}) \to 0
\]
The zig-zags between the middle and right nonzero terms are quasi-isomorphisms by induction. Then the zig-zag on the left is also a quasi-isomorphism by the cohomology long exact sequence and five lemma.

It may be more illuminating to realise the dgas in question as the limits of dgas correponding to simplices $I\in\sK$. Namely, given a subset $J\subset V$, let $\varDelta(J)$ denote a simplex on~$J$, viewed as a simplicial complex on~$V$ (with ghost vertices $V-J$). Then
\[
  R_I(\varDelta(J))=\Lambda[u_i\colon i\notin I]\otimes\Z[v_j\colon j\in J]\big/
  (u_jv_j=v_j^2=0,\;j\in J- I),
\]
Consider the $\ts{cat}^{\mathrm{op}}(\sK)$-diagram
\[
  \mathcal R_{I,\sK}\colon \ts{cat}^{\mathrm{op}}(\sK)\longrightarrow \ts{dga},\quad
  J\longmapsto R_I(\varDelta(J)),
\]
sending a morphism $J_1\subset J_2$ of $\ts{cat}^{\mathrm{op}}(\sK)$ to the surjection of dgas
$R_I(\varDelta(J_2))\to R_I(\varDelta(J_1))$. Then
\[
  R_I(\sK)=\lim\mathcal R_{I,\sK}=\lim_{J\in\sK}R_I(\varDelta(J))
\]
Similarly, we have a $\ts{cat}^{\mathrm{op}}(\sK)$-diagram
\[
  \mathcal C_{I,\sK}\colon \ts{cat}^{\mathrm{op}}(\sK)\longrightarrow \ts{dga},\quad
  J\longmapsto C^*((\!\mb Y,\mb B)^J).
\]
The zig-zag of quasi-isomorphisms~\eqref{zigzagJ} induces an objectwise weak equivalence of diagrams $\mathcal R_{I,\sK}\simeq \mathcal C_{I,\sK}$. The canonical maps $\mathcal R_{I,\sK}(J)\to\lim\mathcal R_{I,\sK}|_{\ts{cat}^{\mathrm{op}}
(\partial\varDelta(J))}$ and $\mathcal C_{I,\sK}(J)\to\lim\mathcal C_{I,\sK}|_{\ts{cat}^{\mathrm{op}}(\partial\varDelta(J))}$ are fibrations (surjections of dgas). Therefore, both diagrams $\mathcal R_{I,\sK}$ and $\mathcal C_{I,\sK}$ are Reedy fibrant (see~\cite[Appendix~C.1]{bu-pa15}). Their limits are therefore quasi-isomorphic. Thus, we obtain the required zig-zag of quasi-isomorphisms of dgas
\[
  R_I(\sK)=\lim_{J\in\sK}R_I(\varDelta(J))\simeq
  \lim_{J\in\sK}C^*((\!\mb Y,\mb B)^J)\xleftarrow\simeq 
  C^*(\mathop{\mathrm{colim}}\limits_{J\in\sK}(\!\mb Y,\mb B)^J)=C^*((\!\mb Y,\mb B)^\sK),
\]
where the second-to-last map is a quasi-isomorphism by excision (or by Mayer--Vietoris).
\end{proof}

For the equivariant cohomology, we obtain

\begin{theorem}\label{cohomology}
There are isomorphisms of rings
\begin{align*}
  H^*_{T_I}(\zk)&\cong H^*\bigl(\Lambda[u_i\colon i\notin I]\otimes\Z[\sK],d\bigr)
  \cong H^*\bigl(R_I(\sK),d\bigr)\\
  &\cong\Tor_{\Z[v_1,\ldots,v_m]}\bigl(\Z[v_i\colon i\in I],\Z[\sK]\bigr),
\end{align*}
where $\Z[v_i\colon i\in I]$ is the $\Z[v_1,\ldots,v_m]$-module via the homomorphism sending $v_i$ to $0$ for $i\notin I$.
\end{theorem}
\begin{proof}
The first two isomorphisms follow from Theorem~\ref{quism}. For the last one, consider the Koszul resolution $\Lambda[u_i\colon i\notin I]\otimes\Z[v_i\colon i\notin I]\to\Z$ of the augmentation $\Z[v_i\colon i\notin I]$-module~$\Z$. Tensoring it with $\Z[v_i\colon i\in I]$ we obtain a free resolution of the $\Z[v_1,\ldots,v_m]$-module 
$\Z[v_i\colon i\in I]$:
\[
  \Lambda[u_i\colon i\notin I]\otimes\Z[v_1,\ldots,v_m]\to\Z[v_i\colon i\in I].
\]
Then $\Tor_{\Z[v_1,\ldots,v_m]}(\Z[v_i\colon i\in I],\Z[\sK])$ is the cohomology of the complex obtained by applying $\otimes_{\Z[v_1,\ldots,v_m]}\Z[\sK]$ to the resolution above, which gives $\Lambda[u_i\colon i\notin I]\otimes\Z[\sK]$.
\end{proof}

When $I=[m]$, we obtain that the singular cochain algebra of $ET^m\times_{T^m}\zk\simeq(\C P^\infty,\pt)^\sK$ is quasi-isomorphic to $\Z[\sK]$ with zero differential, which is the integral formality result of~\cite{no-ra05}.

When $I=\varnothing$, we obtain the description of the ordinary integral cohomology of 
$\zk$ of~\cite{b-b-p04} and~\cite{fran06}.

\section{Equivariant formality}
A $T^k$-space $X$ is called \emph{equivariantly formal} if $H^*_{T^k}(X)$ is  free as a module over $H^*_{T^k}(\pt)=H^*(BT^k)$. The latter condition implies that the spectral sequence of the bundle $ET^k\times_{T^k} X\to BT^k$ collapses at the $E_2$ page.

Using the results of the previous section, we obtain that $\zk$ is equivariantly formal with respect to the action of $T_I$ if $\Tor_{\Z[v_1,\ldots,v_m]}(\Z[v_i\colon i\in I],\Z[\sK])$ is free as a module over $H^*(BT_I)=\Z[v_i\colon i\in I]$.

\begin{lemma}\label{simplex}
Let $\sK=\varDelta[m]$. Then $H^*_{T_I}(\zk)$ is free as a $H^*(BT_I)$-module, for any $I\subset[m]$.
\end{lemma}
\begin{proof}
For $\sK=\varDelta[m]$, we have $\zk\cong D^{2m}$ is $T_I$-equivariantly contractible. Hence, $H^*_{T_I}(\zk)\cong H^*_{T_I}(\pt)= H^*(BT_I)$ is  a free $H^*(BT_I)$-module.
\end{proof}

\begin{lemma}\label{boundary}
Let $\sK=\partial\varDelta[m]$, the boundary of a simplex on~$[m]$. Then $\zk\cong S^{2m-1}$ and $H^*_{T_I}(\zk)$ is free as a $H^*(BT_I)$-module, for any $I\subsetneq[m]$.
\end{lemma}
\begin{proof}
Consider the spectral sequence of the bundle $ET_I\times_{T_I}\zk\to BT_I$ with fibre $\zk\cong S^{2m-1}$. We claim that the homomorphism $H^*(ET_I\times_{T_I}\zk)\to H^*(\zk)$ induced by the inclusion of the fibre is surjective. Indeed, by the construction of the previous section, $H^*(ET_I\times_{T_I}\zk)\to H^*(\zk)$ is the cohomology homomorphism induced by the dga map
\[
  \bigl(\Lambda[u_i\colon i\notin I]\otimes\Z[\sK],d\bigr)\to
  \bigl(\Lambda[u_1,\ldots,u_m]\otimes\Z[\sK],d\bigr).
\]
We have $H^*(\zk)=\Z\langle1,[u_iv_1\cdots \widehat v_{i}\cdots v_m]\rangle$, where $[u_iv_1\cdots \widehat v_{i}\cdots v_m]\in H^{2m-1}(\zk)$ denotes the cohomology class of the cocycle $u_iv_1\cdots \widehat v_{i}\cdots v_m$ with $v_{i}$ omitted (note that $\Z[\sK]=\Z[v_1,\ldots,v_m]/(v_1\cdots v_m)$). Choosing $i\notin I$ we get that $[u_iv_1\cdots \widehat v_{i}\cdots v_m]$ also represents a nontrivial cohomology class in $H^*(ET_I\times_{T_I}\zk)$ (here we use the fact that $I\ne[m]$). Hence, $H^*(ET_I\times_{T_I}\zk)\to H^*(\zk)$
is surjective.

Now $H^q(ET_I\times_{T_I}\zk)\to H^q(\zk)$ is the edge homomorphism 
\[
  H^*(ET_I\times_{T_I}\zk)\to E_\infty^{0,q}
  \to E_2^{0,q}= H^q(\zk)
\]
of the spectral sequence. Its surjectivity implies $E_2^{0,q}=E_\infty^{0,q}$, that is, all the differentials from the first column are trivial. By the multiplicative structure in the spectral sequence, all other differentials are also trivial. We obtain $H^*(ET_I\times_{T_I}\zk)\cong E_\infty=E_2\cong H^*(BT_I)\otimes H^*(\zk)$, a free $H^*(BT_I)$-module.
\end{proof}

Let $\sK_1$ and $\sK_2$ be simplicial complexes on the sets $V_1$ and $V_2$, respectively. Their \emph{join} is the simplicial complex on $V_1\sqcup V_2$ given by
\[
  \sK_1\mathbin{*}\sK_2=\{I_1\sqcup I_2\subset V_1\sqcup V_2\colon I_1\in\sK_1,\;I_2\in\sK_2\}.
\]

\begin{lemma}\label{tensor}
Let $I_1\subset V_1$, $I_2\subset V_2$, $V=V_1\sqcup V_2$, $I=I_1\sqcup I_2$ and $\sK=\sK_1\mathbin{*}\sK_2$.
Suppose that $H^*_{T_{I_1}}(\mathcal Z_{\sK_1})$
is free as a $H^*(BT_{I_1})$-module, and $H^*_{T_{I_2}}(\mathcal Z_{\sK_2})$
is free as a $H^*(BT_{I_2})$-module. Then $H^*_{T_I}(\zk)$ is free as a $H^*(BT_I)$-module.
\end{lemma}
\begin{proof}
We have $\zk\cong\mathcal Z_{\sK_1}\times\mathcal Z_{\sK_2}$ by~\cite[Proposition~4.1.3]{bu-pa15}. Then
\begin{multline*}
H^*_{T_I}(\zk) \cong
H^*_{T_I}(\mathcal Z_{\mathcal{K}_1} \times \mathcal Z_{\mathcal{K}_2}) = 
H^*\bigl(ET_I \times_{T_I} 
(\mathcal Z_{\mathcal{K}_1} \times 
\mathcal Z_{\mathcal{K}_2})\bigr) \\
\cong H^*\bigl((ET_{I_1} \times_{T_{I_1}} \times \mathcal Z_{\mathcal{K}_1}) \times (ET_{I_2} \times_{T_{I_2}} \times 
\mathcal Z_{\mathcal{K}_2})\bigr)\\
\cong H^*(ET_{I_1} \times_{T_{I_1}} \mathcal Z_{\mathcal{K}_1}) \otimes H^*(ET_{I_2} \times_{T_{I_2}} \mathcal Z_{\mathcal{K}_2})
= H^*_{T_{I_1}}(\mathcal Z_{\mathcal{K}_1})\otimes H^*_{T_{I_2}}(\mathcal Z_{\mathcal{K}_2}).
\end{multline*}
The second-to-last isomorphism above follows by the K\"unneth formula, because $H^*(ET_{I_1} \times_{T_{I_1}} \mathcal Z_{\mathcal{K}_1})$ is a free $\Z$-module (as it is a free $\mathbb{Z}[v_i\colon i\in I_1]$-module). The claim follows, since $H^*(BT_I)=H^*(BT_{I_1})\otimes H^*(BT_{I_2})$.
\end{proof}

\begin{lemma}\label{freesim}
Let $I\notin\sK$. Then $H^*_{T_I}(\zk)$ is not free as a module over $H^*(BT_I)$.
\end{lemma}
\begin{proof}
Take $v_I=\prod_{i\in I}v_i\in H^*(BT_I)$. Then $v_I\cdot 1=[v_I]=0$, because $v_I$ represents zero in
$H^*_{T_I}(\zk)=H(\Lambda[u_i\colon i\notin I]\otimes\Z[\sK])$. Hence, $1$ is a $H^*(BT_I)$-torsion element, and $H^*_{T_I}(\zk)$ is not free as a $H^*(BT_I)$-module.
\end{proof}

Let $\sK$ be a simplicial complex on $V$ and $V'\subset V$. The subcomplex $\sK'=\{I\in\sK\colon I\subset V'\}$ is called a \emph{full subcomplex} (an \emph{induced subcomplex} on~$V'$). Equivalently, $\sK'\subset\sK$ is a full subcomplex if any missing face of $\sK'$ is a missing face of~$\sK$.

\begin{lemma}\label{property}
If $H^*_{T_I}(\zk)$ is a free $H^*(BT_I)$-module and $\sK'$ is a full subcomplex of $\sK$ such that $I\in\sK'$, then
$H^*_{T_I}(\mathcal Z_{\mathcal K'})$ is also a free $H^*(BT_I)$-module.
\end{lemma}
\begin{proof}
Since $\sK'$ is a full subcomplex, there is a retraction $\mathcal Z_{\mathcal K'}\to\zk\to\mathcal Z_{\mathcal K'}$ (see~\cite[Proposition~2.2]{pa-ve16} or~\cite[Lemma~4.2]{pa-th19}), which is $T_I$-equivariant for any $I\subset V'$. It follows that $H^*_{T_I}(\mathcal Z_{\mathcal K'})$ is a direct summand in the free $H^*(BT_I)$-module $H^*_{T_I}(\zk)$. Hence, $H^*_{T_I}(\mathcal Z_{\mathcal K'})$ is also free.
\end{proof}

The equivariant cohomology $H^*_{T_I}(\mathcal{Z}_\mathcal{K})$ may fail to be free as a $H^*(BT_I)$-module even when $I$ is a simplex of~$\sK$:

\begin{example}\label{nfex}
Let $\sK$ be an $m$-cycle (the boundary of an $m$-gon), with
vertices numbered counter-clockwise. Let $I=\{i\}$, so that $T_I$ is the $i$th coordinate circle $S^1_i$. When $m=3$ or $m=4$,
$H^*_{S^1_i}(\zk)$ is free over $\Z[v_i]$ for all~$i$ by Lemma~\ref{boundary} and Lemma~\ref{tensor}. Suppose that
$m\ge 5$. Then the nonzero cohomology class in $H^3_{S^1_m}(\zk)$ represented by the cocycle
$u_1v_3\in\Lambda[u_1,\ldots,u_{m-1}]\otimes\Z[\sK]$ is a
$\Z[v_m]$-torsion element. Indeed, $v_m\cdot
[u_1v_3]=[u_1v_3v_m]=0$, since $v_3v_m=0$ in $\Z[\sK]$ for $m\ge5$. Hence, $H^*_{S^1_m}(\zk)$ is not free as a $\Z[v_m]$-module.
\end{example}

Recall that a \emph{missing face} (a \emph{minimal non-face}) of a simplicial complex $\sK$ on $V$ is a subset $I\subset V$ such that $I\notin\sK$ but every proper subset of $I$ is in~$\sK$. In other words, $I$ is a missing face if $\partial\varDelta(I)$ is a subcomplex of $\sK$, but $\varDelta(I)$ is not. 
We denote by $\MF(\sK)$ the set of missing faces of~$\sK$.

Generalising Example~\ref{nfex}, we have

\begin{lemma}\label{condition}
Let $I_1$ and $I_2$ be missing faces of $\sK$, and suppose that $I=I_1-I_2$ is nonempty. Then $H^*_{T_I}(\zk)$ is not free as a module over $H^*(BT_I)$.
\end{lemma}
\begin{proof}
Since $I_1$ and $I_2$ are distinct missing faces, we have $I_2\not\subset I_1$. Take $j\in I_2-I_1$. Then $j\notin I$. The cocycle $u_jv_{I_2-j}$ represents a nontrivial cohomology class in
$H^*_{T_I}(\zk)=H(\Lambda[u_i\colon i\notin I]\otimes\Z[\sK])$.

We claim that the cohomology class $[u_jv_{I_2-j}]\in H^*_{T_I}(\zk)$ is a $H^*(BT_I)$-torsion element. Indeed, take $v_I=\prod_{i\in I}v_i\in H^*(BT_I)$. Then
\[
  v_I\cdot[u_jv_{I_2-j}]=[u_jv_Iv_{I_2-j}]=[u_jv_{I_1}v_{I_2-I_1-j}]=0,
\]
since $v_{I_1}=0$ in $\Z[\sK]$. Hence, $H^*_{T_I}(\zk)$ is not free as a $H^*(BT_I)$-module.
\end{proof}

\begin{theorem}\label{suffic}
Let $\sK$ be a simplicial complex on a finite set $V$. The following conditions are equivalent:

\begin{itemize}
\item[(a)] For any $I\in\sK$, the equivariant cohomology $H^*_{T_I}(\mathcal{Z}_\mathcal{K})$ is a free module over $H^*(BT_I)$.

\smallskip

\item[(b)] There is a partition $V=V_1\sqcup\cdots\sqcup V_p\sqcup U$ such that
\[
  \sK=\partial\varDelta(V_1)\mathbin{*}\cdots
  \mathbin{*}\partial\varDelta(V_p)\mathbin{*}
  \varDelta(U),
\]
where $\varDelta(U)$ denotes a full simplex on $U$, and $\partial\varDelta(V_i)$ denotes the boundary of a simplex on~$V_i$. 

\smallskip

\item[(c)] The rational face ring $\mathbb Q[\sK]$ is a complete intersection ring (the quotient of the polynomial ring by an ideal generated by a regular sequence).
\end{itemize}
\end{theorem}
\begin{proof}
(a)$\Rightarrow$(b) We have $\Q[\sK]=\Q[v_1,\ldots,v_m]/(t_1,\ldots,t_p)$ where  $t_k=\prod_{i\in V_k}v_i$ is a square-free monomial and $V_k$ is a missing face of $\sK$, for $k=1,\ldots,p$. Suppose some of these missing faces intersect nontrivially, say, $V_1\cap V_2\ne\varnothing$. Then $I=V_1-V_2$ is nonempty, and $H^*_{T_I}(\zk)$ is not a free $H^*(BT_I)$-module by Lemma~\ref{condition}. A contradiction.  Hence, $V_1,\ldots,V_p$ are pairwise non-intersecting, so $\sK$ is as described in~(b).

\smallskip

(b)$\Rightarrow$(a) Write $I=I_1\sqcup\cdots\sqcup I_p\sqcup J$, where $I_k\subsetneq V_k$, $J\subset U$. Then $H^*_{T_I}(\mathcal{Z}_\mathcal{K})$ is a free $H^*(BT_I)$-module
by Lemmas~\ref{simplex}, \ref{boundary} and~\ref{tensor}.

\smallskip

(b)$\Rightarrow$(c) Recall~\cite[\S A.3]{bu-pa15} that a sequence of homogeneous elements $(t_1,\ldots,t_k)$ of positive degree in $\Q[v_1,\ldots,v_m]$ is a \emph{regular sequence} if $t_{i+1}$ is not a zero divisor in $\Q[v_1,\ldots,v_m]/(t_1,\ldots,t_i)$ for $0\le i< k$. If $\sK$ is as in~(b), then $\Q[\sK]=\Q[v_1,\ldots,v_m]/(t_1,\ldots,t_p)$, where $m=|V|$ and $t_k=\prod_{i\in V_k}v_i$ for $k=1,\ldots,p$. Then $(t_1,\ldots,t_p)$ is a regular sequence, so $\Q[\sK]$ is a complete intersection ring.

\smallskip

(c)$\Rightarrow$(b) Suppose $\Q[\sK]=\Q[v_1,\ldots,v_m]/(t_1,\ldots,t_p)$ where $(t_1,\ldots,t_p)$ is a regular sequence. We can assume that $t_k=\prod_{i\in V_k}v_i$ where $V_k$ is a missing face of $\sK$, for $k=1,\ldots,p$. Suppose some of these missing faces intersect nontrivially, say, $V_1\cap V_2\ne\varnothing$. Then $t_2\cdot\prod_{i\in V_1-V_2}v_i=t_1\cdot\prod_{j\in V_2-V_1}v_j$, so $t_2$ is a zero divisor in $\Q[v_1,\ldots,v_m]/(t_1)$. A contradiction. Hence, $V_1,\ldots,V_p$ are pairwise non-intersecting, so $\sK$ is as described in~(b).
\end{proof}
 
The equivalence (b)$\Leftrightarrow$(c) of Theorem~\ref{suffic} was noted in~\cite[\S5]{no-ra05}.

Recall that $\sK$ is called a \emph{flag complex} if each of its missing faces has two vertices. A simplicial complex $\sK$ is flag if and only if it has no ghost vertices and any set of vertices of $\sK$ which are pairwise connected by edges spans a simplex. In the case of flag complexes we have the following specification of the criterion in Theorem~\ref{suffic}.

\begin{theorem}\label{flag}
Let $\mathcal{K}$ be a flag complex on~$V$. Then the following conditions are equivalent:
\begin{itemize}
\item[(a)]  $H^*_{S^1_i}(\mathcal{Z}_\mathcal{K})$ is a free module over $\Z[v_i]$ for all $i$.

\item[(b)]$\sK=\partial\varDelta(V_1)\mathbin{*}\cdots\mathbin{*}
\partial\varDelta(V_p)\mathbin{*} \varDelta(U)$ where $|V_k|=2$ for $k=1,\ldots,p$.
\end{itemize}
\end{theorem}

\begin{proof}
Implication (b)$\Rightarrow$(a) follows from Theorem~\ref{suffic}, so we only need to prove (a)$\Rightarrow$(b). Let $V_1$, $V_2$ be missing faces. Then $|V_1|=|V_2|=2$. If $V_1\cap V_2\ne\varnothing$, then $V_1- V_2=\{i\}$ for some $i\in V$. Then $H^*_{S^1_i}(\mathcal{Z}_\mathcal{K})$ is not free as a 
$\Z[v_i]$-module by Lemma~\ref{condition}. This contradiction shows that all missing faces of $\sK$ are pairwise non-intersecting, so $\sK$ is as in~(b).
\end{proof}

Here is an example showing that the equivalence of Theorem~\ref{flag} does not hold in the non-flag case.

\begin{example}
Let $\sK$ be the simplicial complex on $5$ vertices with $\MF(\sK) = \{I_1, I_2\}$, where $I_1 = \{1, 2, 3\}$ and $I_2 = \{3, 4, 5\} $.
Then $H^*_{T_I}(\zk)$ is not free as a $H^*(BT_I)$-module for $I = \{1, 2\}$ (or for $I = \{4, 5\}$) due to Lemma~\ref{condition}. 
However, $H^*_{S^1_i}(\zk)$ is a free $H^*(BS^1_i)$-module for all~$i$.
Indeed, it can be shown by the methods of~\cite[\S8]{gr-th16} or~\cite{ab-pa19} that $\zk \cong S^5 \vee S^5 \vee S^8$. The ordinary cohomology is generated by the following monomials in the Koszul algebra $\Lambda[u_1,\ldots,u_5]\otimes\Z[\sK]$:
\[
  H^*(\zk) = \Z\langle1,[u_k v_{I_1-k}], 
  [u_j v_{I_2-j}], [u_{i_1} u_{i_2} v_{[5]-i_1-i_2}]\rangle,
\]
where $k\in I_1$, $j\in I_2$, $[u_k v_{I_1-k}], [u_j v_{I_2-j}]\in H^5(\zk)$, $i_1\in I_1 - I_2$, $i_2\in I_2 - I_1$ and $[u_{i_1} u_{i_2} v_{[5]-i_1-i_2}]\in H^8(\zk)$.

Since both $I_1-I_2$ and $I_2-I_1$ contain two elements, we can choose $k,j,i_1,i_2$ such that
$i\notin \{k, j, i_1, i_2\}$. Then the monomials $u_k v_{I_1-k}$, $u_j v_{I_2-j}$, $u_{i_1} u_{i_2} v_{[5]-i_1-i_2}$ represent nontrivial cohomology classes in $H^*_{S^1_i}(\zk)$. This implies that the homomorphism $H^*_{S^1_i}(\zk)\to H^*(\zk)$ is surjective. Therefore, the spectral sequence of the bundle $ES^1_i \times_{S^1_i} \zk \to BS^1_i$ collapses at the $E_2$ page, as in the proof of Lemma~\ref{boundary}. It follows that $H^*_{S^1_i}(\zk) \cong H^*(BS^1_i) \otimes H^*(\zk)$, a free $H^*(BS^1_i)$-module.
\end{example}

The equivalence similar to that of Theorem~\ref{flag} also holds when $\sK$ is one-dimensional.

\begin{theorem}\label{graph}
Let $\mathcal{K}$ be a one-dimensional complex (a simple graph). Then the following conditions are equivalent: 
\begin{itemize}
\item[(a)]
$H^*_{S^1_i}(\mathcal{Z}_{\mathcal{K}})$ is a free module over $\mathbb{Z}[v_i]$ for any $i$.
\item[(b)]
$\mathcal{K}$ is the one of the following: $\partial\varDelta^2$, $\partial\varDelta^1\mathbin{*}\partial\varDelta^1$, $\partial\varDelta^1$, $\varDelta^1$, $\partial\varDelta^1*\varDelta^0$, $\varDelta^0$.
\end{itemize}
\end{theorem}
\begin{proof}
Implication (b)$\Rightarrow$(a) follows from Theorem~\ref{suffic}, so we only need to prove (a)$\Rightarrow$(b). We consider several cases.

\smallskip

\emph{Case 1: $\sK$ is a tree.} If it has no more than three vertices, then $\sK$ is $\varDelta^1$, $\partial\varDelta^1*\varDelta^0$ or $\varDelta^0$. In each of these cases $H^*_{S^1_i}(\mathcal{Z}_{\mathcal{K}})$ is a free $\mathbb{Z}[v_i]$-module by Theorem \ref{suffic}. 


Suppose $\sK$ has more than three vertices. Then $\sK$ has a connected induced subgraph $\sK_1$ on $4$ vertices, which has the form $\quad\begin{picture}(12,5) \put(0,0){\circle*{1}}
\put(5,0){\circle*{1}} \put(10,0){\circle*{1}}
\put(5,5){\circle*{1}} \put(0,0){\line(1,0){10}}
\put(5,0){\line(0,1){5}}
\end{picture}$ or 
$\;\begin{picture}(15,2)
\put(0,1){\circle*{1}} \put(5,1){\circle*{1}}
\put(10,1){\circle*{1}} \put(15,1){\circle*{1}}
\put(0,1){\line(1,0){15}}\end{picture}\quad$. In both cases, there are $I_1, I_2 \in\MF(\sK_1)$ such that $I_1 - I_2 = \{i\}$ for some~$i$. 
Then $H^*_{S^1_i}(\mathcal Z_{\mathcal K_1})$ is not free over $\Z[v_i]$ by Lemma~\ref{condition}, and $H^*_{S^1_i}(\zk)$ is also not free by Lemma~\ref{property}. A contradiction. 

\smallskip

\emph{Case 2: $\sK$ is a disjoint union of trees.} If $\sK$ has two vertices, then $\sK = \partial\varDelta^1$.

Suppose $\sK$ has more than two vertices. Write $\sK = \sK_1 \sqcup \cdots \sqcup \sK_s$ where each $\sK_j$ is a tree. Then each $\sK_j$ has at most three vertices by Case~1. Take $I_1 = \{i, j\}$, $I_2 = \{k, j\}$, where $i\in \sK_1$, $j\in \sK_2$ and $\{k, j\} \notin \sK$. Then $I_1,I_2\in\MF(\sK)$ and $I_1 - I_2 = \{i\}$. Hence, $H^*_{S^1_i}(\zk)$ is not a free $\Z[v_i]$-module by Lemma \ref{condition}. A contradiction.

\smallskip

\emph{Case 3: $\sK$ has a $3$-cycle.}
If $\sK$ is a $3$-cycle, then $\sK=\partial\varDelta^2$. 

Suppose $\sK$ has at least $4$ vertices. Consider the induced subgraph on $4$ vertices containing a $3$-cycle. There are four cases:
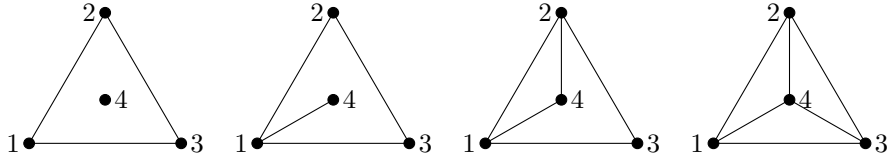
\[
\begin{tikzpicture}
\pgfmathsetmacro{\w}{1};
\pgfmathsetmacro{\mar}{1};
\pgfmathsetmacro{\h}{1/sqrt(3)};
\pgfmathsetmacro{\skip}{\mar + 2*\w};
\draw 
(0, 0) coordinate[] (a) -- 
(\w, 3*\h) coordinate[] (b) -- 
(2*\w, 0) coordinate[] (c) -- cycle;
\filldraw (0,0) circle (2pt) node[anchor=east,black]{1};
\filldraw (\w,3*\h) circle (2pt) node[anchor=east,black]{2};
\filldraw (2*\w,0) circle (2pt) node[anchor=west,black]{3};
\filldraw (\w,\h) circle (2pt) node[anchor=west,black]{4};
\draw 
(\skip, 0) coordinate[] (a) -- 
(\skip+\w, 3*\h) coordinate[] (b) -- 
(\skip+2*\w, 0) coordinate[] (c) -- cycle;
\draw (\skip, 0) coordinate[] (a) -- (\skip + \w, \h) coordinate[] (d);
\filldraw (\skip,0) circle (2pt) node[anchor=east,black]{1};
\filldraw (\skip+\w,3*\h) circle (2pt) node[anchor=east,black]{2};
\filldraw (\skip+2*\w,0) circle (2pt) node[anchor=west,black]{3};
\filldraw (\skip+\w,\h) circle (2pt) node[anchor=west,black]{4};
\draw 
(2*\skip, 0) coordinate[] (a) -- 
(2*\skip+\w, 3*\h) coordinate[] (b) -- 
(2*\skip+2*\w, 0) coordinate[] (c) -- cycle;
\draw (2*\skip, 0) coordinate[] (a) -- (2*\skip + \w, \h) coordinate[] (d);
\draw (2*\skip+\w, 3*\h) coordinate[] (b) -- (2*\skip + \w, \h) coordinate[] (d);
\filldraw (2*\skip,0) circle (2pt) node[anchor=east,black]{1};
\filldraw (2*\skip+\w,3*\h) circle (2pt) node[anchor=east,black]{2};
\filldraw (2*\skip+2*\w,0) circle (2pt) node[anchor=west,black]{3};
\filldraw (2*\skip+\w,\h) circle (2pt) node[anchor=west,black]{4};
\draw 
(3*\skip, 0) coordinate[] (a) -- 
(3*\skip+\w, 3*\h) coordinate[] (b) -- 
(3*\skip+2*\w, 0) coordinate[] (c) -- cycle;
\draw (3*\skip, 0) coordinate[] (a) -- (3*\skip + \w, \h) coordinate[] (d);
\draw (3*\skip+\w, 3*\h) coordinate[] (b) -- (3*\skip + \w, \h) coordinate[] (d);
\draw (3*\skip+2*\w, 0) coordinate[] (b) -- (3*\skip + \w, \h) coordinate[] (d);
\filldraw (3*\skip,0) circle (2pt) node[anchor=east,black]{1};
\filldraw (3*\skip+\w,3*\h) circle (2pt) node[anchor=east,black]{2};
\filldraw (3*\skip+2*\w,0) circle (2pt) node[anchor=west,black]{3};
\filldraw (3*\skip+\w,\h) circle (2pt) node[anchor=west,black]{4};
\end{tikzpicture}
\]
In the first two cases, take $I_1 = \{3, 4\}$, $I_2 = \{2, 4\}$ in $\MF(\sK)$. In the last two cases, take $I_1 = \{1, 2, 3\}$ and $I_2 = \{1, 2, 4\}$ in $\MF(\sK)$. Then $I_1 - I_2 = \{3\}$ and $H^*_{S^1_3}(\zk)$ is not a free $\Z[v_3]$-module by Lemma~\ref{condition}. A contradiction.

\smallskip

\emph{Case 4: $\sK$ has no $3$-cycles and has a $4$-cycle.} If $\sK$ is a $4$-cycle, then $\sK = \partial\varDelta^1 * \partial\varDelta^1$. 

Suppose $\sK$ has more than $4$ vertices. Consider the induced subgraph on $5$ vertices containing a $4$-cycle. Since there are no $3$-cycles, there are three cases:
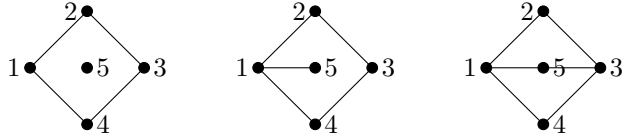
\[
\begin{tikzpicture}
\pgfmathsetmacro{\w}{0.75};
\pgfmathsetmacro{\mar}{0.75};
\pgfmathsetmacro{\h}{0.75};
\pgfmathsetmacro{\skip}{\mar + 3*\w};
\draw 
(0, \h) coordinate[] (a) -- 
(\w, 2*\h) coordinate[] (b) -- 
(2*\w, \h) coordinate[] (c) -- 
(\w, 0) coordinate[] (d) -- cycle;
\filldraw (0,\h) circle (2pt) node[anchor=east,black]{1};
\filldraw (\w,2*\h) circle (2pt) node[anchor=east,black]{2};
\filldraw (2*\w,\h) circle (2pt) node[anchor=west,black]{3};
\filldraw (\w,0) circle (2pt) node[anchor=west,black]{4};
\filldraw (\w,\h) circle (2pt) node[anchor=west,black]{5};

\draw 
(\skip, \h) coordinate[] (a) -- 
(\skip + \w, 2*\h) coordinate[] (b) -- 
(\skip+2*\w, \h) coordinate[] (c) -- 
(\skip+\w, 0) coordinate[] (d) -- cycle;
\draw (\skip, \h) coordinate[] (a) -- (\skip + \w, \h) coordinate[] (e);
\filldraw (\skip,\h) circle (2pt) node[anchor=east,black]{1};
\filldraw (\skip+\w,2*\h) circle (2pt) node[anchor=east,black]{2};
\filldraw (\skip+2*\w,\h) circle (2pt) node[anchor=west,black]{3};
\filldraw (\skip+\w,0) circle (2pt) node[anchor=west,black]{4};
\filldraw (\skip+\w,\h) circle (2pt) node[anchor=west,black]{5};

\draw 
(2*\skip, \h) coordinate[] (a) -- 
(2*\skip + \w, 2*\h) coordinate[] (b) -- 
(2*\skip+2*\w, \h) coordinate[] (c) -- 
(2*\skip+\w, 0) coordinate[] (d) -- cycle;
\draw (2*\skip, \h) coordinate[] (a) -- (2*\skip + \w, \h) coordinate[] (e);
\draw (2*\skip + 2*\w, \h) coordinate[] (b) -- (2*\skip + \w, \h) coordinate[] (e);
\filldraw (2*\skip,\h) circle (2pt) node[anchor=east,black]{1};
\filldraw (2*\skip+\w,2*\h) circle (2pt) node[anchor=east,black]{2};
\filldraw (2*\skip+2*\w,\h) circle (2pt) node[anchor=west,black]{3};
\filldraw (2*\skip+\w,0) circle (2pt) node[anchor=west,black]{4};
\filldraw (2*\skip+\w,\h) circle (2pt) node[anchor=west,black]{5};
\end{tikzpicture}
\]
In all cases take $I_1 = \{2, 4\}$, $I_2 = \{4, 5\}$ in $\MF(\sK)$, then $I_1 - I_2 = \{2\}$ and $H^*_{S^1_2}(\zk)$ is not a free $\Z[v_2]$-module by Lemma~\ref{condition}. A contradiction again.

\smallskip

\emph{Case 5: each minimal cycle in $\sK$ has length at least~$5$.} Then
$\sK$ has an induced subgraph $\sK_1$ which is an $m$-cycle with $m\ge 5$. As in Example~\ref{nfex}, we have that $H^*_{S^1_m}(\mathcal Z_{\mathcal K_1})$ is not free as a $\Z[v_m]$-module. So $H^*_{S^1_m}(\zk)$ is also not free by Lemma~\ref{property}. A contradiction.
\end{proof}

\end{document}